\documentclass{amsart}

\newtheorem{theorem}{Theorem}[section]
\newtheorem{lemma}[theorem]{Lemma}

\theoremstyle{definition}
\newtheorem{definition}[theorem]{Definition}
\newtheorem{example}[theorem]{Example}

\theoremstyle{remark}
\newtheorem{remark}[theorem]{Remark}

\numberwithin{equation}{section}

\begin{document}

\title{A new formula for $\zeta(s)$}

\author{Chenfeng He}

\address{Department of Mathematics, Zhejiang University, Hangzhou, 310027 China }

\begin{abstract}
In this paper, by introducing a new operation in the vector space of analytic functions, the author presents a method for derivating the well-known formulas: $\zeta(1-k)=-\frac{B_k}{k}$ and $\zeta(1-n,a)=-\frac{B_n(a)}{n}$ , where $\zeta$, $\zeta(1-n,a)$ denote the Riemann zeta function and  the Hurwitz zeta function respectively. $B_k$ is the $k$-th Bernoulli number. Also the author steps further to deduce some identities related to Bernoulli number and Bernoulli polynomial. Moreover, when combining the operation with forward difference, we can show  a new formula for Riemann zeta function, i.e.
\[\zeta(s)=e\sum_{n=0}^{\infty}\sum_{i=0}^{n}(-1)^{n-i}\frac{1}{(n-i)!(1+i)^{s}}.\] \\
Keywords: Bernoulli number, Riemann zeta function, finite difference, convolution.
\end{abstract}
\maketitle

\section{introduction}
The Bernoulli numbers are rooted in the early history of the computation of sums of integer powers. They are mysterious and very important in mathematics. We will define a new operation motivated by \cite{4} so that it may unveil a part of its mystery.

	Let's consider the vector space of all analytic functions defined in a open set containing $0$,  and each of its element can  be uniquely represented as a Taylor series $\sum_{n=0}^{\infty}a_n\frac{z^{n}}{n!}$. If we consider the set \{$\frac{z^n}{n!}\mid n=0,1,2,...$\} to be a basis of this vector space over $\mathbb{C}$, then we can use a sequence $ (a_0,a_1,a_2,\cdots)$ denoted $\boldsymbol{a}$ to replace the function.  Also for the $n$-th component we use ($\boldsymbol{a})_n=a_n$ to pick it out. In this vector space addtion and scalar multiplication are defined as usual, which is $\boldsymbol{a} + \boldsymbol{b}=(a_0+b_0,a_1+b_1,a_2+b_2,\cdots), \ \lambda\boldsymbol{a}=(\lambda a_0,\lambda a_1,\lambda a_2,\cdots)$.
\begin{definition}
We introduce an operation between two vectors called convolution denoted  $\star$, and define it by its $n$-th coefficient  \[(\boldsymbol{a}\star\boldsymbol{b})_n=\sum^{n}_{k=0}\binom{n}{k}a_kb_{n-k}\]\end{definition}
This can be understood as multiplication of two series or  two analytic functions. We have
\begin{lemma}	
The convolution $\star$ satisfies the following rules:

\begin{enumerate}
\item \textbf{Commutative}: 
$\boldsymbol{a}\star\boldsymbol{b}=\boldsymbol{b}\star\boldsymbol{a}$

\item \textbf{Associative} : $
(\boldsymbol{a}\star\boldsymbol{b})\star\boldsymbol{c}=\boldsymbol{a}\star(\boldsymbol{b}\star\boldsymbol{c})$

\item \textbf{Distributive}: $
\boldsymbol{a}\star(\boldsymbol{b}+\boldsymbol{c})=\boldsymbol{a}\star\boldsymbol{b}+\boldsymbol{a}\star\boldsymbol{c}$

\item $\boldsymbol{a}\star \lambda\boldsymbol{b}=\lambda\boldsymbol{a}\star \boldsymbol{b}=\lambda(\boldsymbol{a}\star\boldsymbol{b})$
\end{enumerate}
\end{lemma}

\begin{proof} (1), (3) and (4) are easily obtained. We only show (2).
\begin{align*}((\boldsymbol{a}\star\boldsymbol{b})\star\boldsymbol{c})_n & = 
\sum_{k=0}^{n}\binom{n}{k}(\sum_{i=0}^{k}\binom{k}{i}a_ib_{k-i})c_{n-k}=\sum_{i=0}^{n}\binom{n}{i}a_i\sum_{k=i}^{n}\binom{n-i}{k-i}b_{k-i}c_{n-k} \\ & = \sum^{n}_{i=0}a_i\sum^{n}_{v=0}\binom{n-i}{v}b_{v}c_{n-i-v}=(\boldsymbol{a}\star(\boldsymbol{b}\star\boldsymbol{c}))_n.\end{align*}
\end{proof}
Define the identity as $\boldsymbol{id}=(1,0,0,\cdots)$. If $\boldsymbol{a}\star\boldsymbol{b}=\boldsymbol{id}$, we say $ \boldsymbol{b}$ is the inverse of $\boldsymbol{a}$. One can easily see that if the first component of a vector is not $0$, then it always has the inverse. If one treats the subset of all vectors with nonzero leading component as a group with the operation $\star $, then it's an abelian group.
 
\begin{example} Now we look at some special vectors in this space. \\
Define $\boldsymbol{j}=(j^0,j^1,j^2,\cdots), \ \boldsymbol{j^{-1}}=((-j)^0,(-j)^1,(-j)^2,\cdots),\ j\in \mathbb{C},\  j\neq0. $ One can see $\boldsymbol{j}=e^{jz}, \  \boldsymbol{j^{-1}}=e^{-jz}$, therefore the inverse of $\boldsymbol{j}$ is $\boldsymbol{j^{-1}}$ under $\star$. For convenience, we write $\boldsymbol{a}\boldsymbol{b}=(a_0b_0,a_1b_1,a_2b_2,\cdots)$. So $\boldsymbol{j^{-1}}=\boldsymbol{-1}\boldsymbol{j}.$
\end{example}
\subsection{Bernoulli number}
Recall that the generating function of Bernoulli numbers is \[\frac{z}{e^z-1}=\sum_{k=0}^{\infty}B_k\frac{z^k}{k!},\] and we denote it as $\boldsymbol{B}=(B_0,B_1,B_2,\cdots)$. 
The inverse of $\boldsymbol{B}$ is \[\frac{e^z-1}{z}=\sum_{k=0}^{\infty}\frac{1}{k+1}\frac{z^k}{k!},\]
 denoted $\boldsymbol{H}=(1,\frac{1}{2},\frac{1}{3},\cdots)$. So we have \begin{equation}\boldsymbol{B}\star\boldsymbol{H}=\boldsymbol{id}.\end{equation} 
 By comparing  components, we get \begin{equation}\sum_{k=0}^{n}\binom{n}{k}B_kH_{n-k}=\sum_{k=0}^{n}\binom{n}{k}B_k\frac{1}{n-k+1}=\delta_{0,n}.\end{equation} 
 Multiply $n+1$ on both sides, and we get  the exact recursive defintion of Bernoulli numbers \begin{equation}\sum_{k=0}^{n}\binom{n+1}{k}B_k=\delta_{0,n}(n+1), \end{equation} where $ \delta$ denotes the Kronecker delta. \\

We also note that \[\frac{-z}{e^{-z}-1}=\frac{ze^{z}}{e^{z}-1},\]which means 
\begin{equation}
\boldsymbol{-1}\boldsymbol{B}=\boldsymbol{B}\star\boldsymbol{1}.\end{equation}
Comparing their $n$-th components, we get
 \begin{equation}\sum_{k=0}^{n}\binom{n}{k}B_k=(-1)^nB_n.\end{equation} 
 Comparing $(1.3)$ and $(1.5)$, we get $B_{2k+1}=0, \  k\geq1$.
 
Considering \begin{equation}
\sum_{\boldsymbol{j=1}}^{\boldsymbol{m}}\boldsymbol{j}=\sum_{j=1}^{m}e^{jz}=e^{z}\frac{e^{mz}-1}{e^{z}-1}=e^{z}\frac{e^{mz}-1}{z}\frac{z}{e^{z}-1}=\boldsymbol{1}\star\boldsymbol{B}\star m\boldsymbol{H}\boldsymbol{m}.\end{equation} From $(1.4)$ we know  \begin{equation} \sum_{\boldsymbol{j=1}}^{\boldsymbol{m}}\boldsymbol{j}=\boldsymbol{-1}\boldsymbol{B}\star m\boldsymbol{H}\boldsymbol{m}.\end{equation} By comparing the components, we get \[\sum_{j=1}^{m}j^{n}=\sum_{k=0}^{n}\binom{n}{k}(-1)^{k}B_k\frac{m^{n-k+1}}{n-k+1}=\frac{1}{n+1}\sum_{k=0}^{n}\binom{n+1}{k}(-1)^{k}B_km^{n-k+1}.\]This is exactly the Faulhaber's Formula.

\subsection{Vector multiplication}
Define the vector multiplication to be \[\boldsymbol{a}\boldsymbol{b}=(a_0b_0,a_1b_1,a_2b_2,\cdots).\] 
Denote $e^{\frac{z}{j}}$ as $\boldsymbol{\frac{1}{j}}=(j^{0},j^{-1},j^{-2},\cdots), \  j \in \mathbb{C},\  j\neq0$. So $\boldsymbol{j}\boldsymbol{\frac{1}{j}}=\boldsymbol{1}.$
\begin{lemma} From the definition we can easily get the following equations:\begin{enumerate}
\item 
$ \boldsymbol{1}\boldsymbol{b}=\boldsymbol{b}$
\item $\boldsymbol{id}\star \boldsymbol{a}=\boldsymbol{a}$
\item 
$	\boldsymbol{j}(\boldsymbol{a}\star \boldsymbol{b})=\boldsymbol{j}\boldsymbol{a}\star\boldsymbol{j}\boldsymbol{b}$
\item $\boldsymbol{-1a}\star\boldsymbol{b}=\boldsymbol{-1}(\boldsymbol{a}\star\boldsymbol{-1b})$
\end{enumerate}
\end{lemma}
From this lemma we know 
\begin{equation}
\boldsymbol{-1B}\star\boldsymbol{-1H}=\boldsymbol{-1}(\boldsymbol{B}\star\boldsymbol{H})=\boldsymbol{-1}\boldsymbol{id}=\boldsymbol{id}.
\end{equation} \\

\section{Appliaction}
In this section we will use the operation $\star$ to derivate some identities related to Bernoulli polynomial and Riemann zeta function.
\subsection{Polynomials}
We know the generating function of Bernoulli polynomial is \[\frac{ze^{xz}}{e^{z}-1}=\sum_{k=0}^{\infty}B_k(x)\frac{z^{k}}{k!}.\] 
This is to say \begin{equation}
\boldsymbol{B} \star \boldsymbol{ x}=(B_0(x),B_1(x),B_2(x),\cdots).\end{equation} 
We denote it as $\boldsymbol{B(x)}$. So the definition of Bernoulli polynomials becomes \[B_n(x)=\sum_{k=0}^{n}\binom{n}{k}B_kx^{n-k}.\]
 Let's find out the generating function of $s_{n}(m)=\sum_{j=1}^{m}j^{n}$, denoted $\boldsymbol{G}$, by using our operation.

First, let's take a look at $\boldsymbol{H}\star\boldsymbol{j}$. Its $n$-th component is \[\sum_{k=0}^{n}\binom{n}{k}\frac{j^{n-k}}{k+1}=\frac{1}{n+1}\sum_{k=0}^{n}\binom{n+1}{k+1}j^{n+1-(k+1)}=\frac{(j+1)^{n+1}-j^{n+1}}{n+1}.\] 
This is to say $ \boldsymbol{H} \star\boldsymbol{ j}=(j+1)\boldsymbol{H(j+1)}-j\boldsymbol{hj}$. 
By summing over $1$ through $m$, we get \[\sum_{\boldsymbol{j=1}}^{\boldsymbol{m}}\boldsymbol{H} \star\boldsymbol{ j}=(m+1)\boldsymbol{H(m+1)}-\boldsymbol{H}.\]
 Using distributive law we obtain
  \[\sum_{\boldsymbol{j=1}}^{\boldsymbol{m}}\boldsymbol{H }\star\boldsymbol{ j}=\boldsymbol{H} \star \sum_{\boldsymbol{j=1}}^{\boldsymbol{m}}\boldsymbol{j}.\]
  Recalling $\boldsymbol{B}\star\boldsymbol{H}=\boldsymbol{id}$, do the operation with $ \boldsymbol{B} $ to both sides, 
  and we get \[\sum_{\boldsymbol{j=1}}^{\boldsymbol{m}}\boldsymbol{j}=\boldsymbol{B}\star (m+1)\boldsymbol{H(m+1)}-\boldsymbol{id}=\frac{z}{e^{z}-1}\frac{e^{(m+1)z}-1}{z}-1.\] 
 Hence the generating function of $s_n(m)$ is \begin{equation}G(z,m)=\frac{e^{(m+1)z}-1}{e^{z}-1}-1=\frac{e^{(m+1)z}-e^{z}}{e^{z}-1}.\end{equation} 
We know $s_n(m)$ can be represented as $\frac{B_{n+1}(m+1)-B_{n+1}(1)}{n+1}
$, our operation can get this as well. \\
Multiplying $z$ on both sides of $(2.2)$, we have 
 \begin{equation}zG(z,m)=\frac{ze^{(m+1)z}}{e^{z}-1}-\frac{ze^{z}}{e^{z}-1}=\boldsymbol{B(m+1)}-\boldsymbol{B(1)},\end{equation}
 which means
  \[(0,s_0(m),2s_1(m),3s_2(m),\cdots)=(B_0(m+1),B_1(m+1),B_2(m+1),\cdots)-(B_0(1),B_1(1),B_2(1),\cdots).\]
   Comparing  the components, we get 
   \begin{equation}s_n(m)=\frac{B_{n+1}(m+1)-B_{n+1}(1)}{n+1}.\end{equation}
 Let's give the last two applications of this part. Although the operation can be used to get a lot of identities, we need to focus on the main goal of this paper.
From \[\boldsymbol{H}\star\boldsymbol{B(x)}=\boldsymbol{H}\star\boldsymbol{B}\star\boldsymbol{x}=\boldsymbol{x},\] we get \begin{equation}\sum_{k=0}^{n}\binom{n}{k}\frac{B_k(x)}{n-k+1}=x^n=\frac{1}{n+1}\sum_{k=0}^{n}\binom{n+1}{k}B_k(x).\end{equation} We also note $\boldsymbol{x}\star \boldsymbol{1}=\boldsymbol{(x+1)}$, then we obtain 
\begin{equation} \boldsymbol{B}\star\boldsymbol{x}\star\boldsymbol{1}=\boldsymbol{B(x)}\star \boldsymbol{1}=\boldsymbol{B(x+1)}.\end{equation}     This is \begin{equation}\sum_{k=0}^{n+1}\binom{n+1}{k}B_k(x)=B_{n+1}(x+1).\end{equation}
Combining $(2.5)$ and $(2.7)$, we get \begin{equation} x^{n}=\frac{B_{n+1}(x+1)-B_{n+1}(x)}{n+1}.\end{equation} 
Let's take a break to see what is happening. The reader may find our $\star$  equivalent to the multiplication of two functions, and yes it is. But when going deeper  you may discover that this operation helps you to X-ray  a given function. For example, if we merely look at the formula $\frac{e^{z}-1}{z}$, hardly do we get the inner information $(1,\frac{1}{2},\frac{1}{3},\cdots)$ as defined. \\ 

Using the generating function of Bernoulli polynomial, we get \[\boldsymbol{B(1-x)}=\boldsymbol{1}\star\boldsymbol{B(-x)}=\boldsymbol{1}\star\boldsymbol{B}\star\boldsymbol{-1x}.\]
 By $(1.4)$ we obtain
  \[\boldsymbol{B(1-x)}=\boldsymbol{-1B}\star\boldsymbol{-1x}=\boldsymbol{-1(B\star x)}=\boldsymbol{-1B(x)},\] 
 which yields $B_n(1-x)=(-1)^{n}B_{n}(x).$ 
 Also  \[\boldsymbol{-1B(-x)}=\boldsymbol{-1(B\star-1x)}=\boldsymbol{-1B \star x}\] and  
 \[\boldsymbol{-1B}=\boldsymbol{B}+(0,1,0,0,\cdots)\] 
 give \[ \boldsymbol{-1B}\star\boldsymbol{x}=\boldsymbol{B}\star\boldsymbol{x}+(0,1,0,0,\cdots)\star\boldsymbol{x}.\] 
 This is to say  $(-1)^{n}B_n(-x)=B_n(x)+nx^{n-1}$.\\

The point of all these applications is to understand what the operation actually is. Now we move on to the next part.
\subsection{$\zeta$-function}
The relation between $\zeta(1-k)$ and $B_k$ is usually derivated by using functional equation of $\zeta(s)$ and Euler's work about $\zeta(2k)$(see\cite{1}). Now we get this in our way.

Recall that Riemann zeta function was defined as \[\zeta(s)=\sum_{n=1}^{\infty}\frac{1}{n^s},\] for $Re(s)>1$,  and could be extended to an analytic function for all $s\in \mathbb{C}$, except for $s=1$, where it has a simple pole. \cite{3} provides an easy way of extension without using functional equation. The conclusion is 
\begin{equation}
\zeta(s)=1+\frac{1}{s-1}-\sum_{r=1}^{m}\frac{s(s+1)\cdots(s+r-1)}{(r+1)!}(\zeta(s+r)-1)-\frac{s(s+1)\cdots(s+m)}{(m+1)!}\sum_{n=1}^{\infty}\int_{0}^{1}\frac{u^{m+1}du}{(u+n)^{s+m+1}}, \end{equation}
and the sum on the right hand converges for $Re(s)>-m$. If we put $s=1-m$ in $(2.9)$ and note that for $r=m,\  \zeta(s+r)$ has a simple pole with residue $1$, we obtain \begin{equation}
\zeta(1-m)=1-\frac{1}{m}+\frac{(-1)^m}{m(m+1)}-\sum_{r=1}^{m-1}(-1)^r\binom{m-1}{r}\frac{1}{r+1}(\zeta(1-m+r)-1).
\end{equation}
Multiplying $m$ to both sides: \begin{equation} m\zeta(1-m)=m-1+\frac{(-1)^m}{m+1}-\sum_{r=1}^{m-1}(-1)^r\binom{m}{r+1}(\zeta(1-m+r)-1)\end{equation} 
\[ =\frac{(-1)^m}{m+1}+\sum_{r=1}^{m-1}(-1)^{r+1}\binom{m}{r+1}\zeta(1+r-m)-(\sum_{v=2}^{m}\binom{m}{v}(-1)^{v}-m+1).\]  Note in addition that \[\sum_{v=2}^{m}\binom{m}{v}(-1)^{v}-m+1=\sum_{v=0}^{m}\binom{m}{v}(-1)^{v}=0,\]
and we get 
\begin{theorem} For integers $m>0$,
\begin{equation}\frac{(-1)^{m+1}}{m+1}=\sum_{i=1}^{m}(-1)^{i}\binom{m}{i}\zeta(i-m).\end{equation} 
\end{theorem}
 If we denote $(\zeta(0),\zeta(-1),\zeta(-2),\cdots)$ as $\boldsymbol{\zeta(-s)}$, and we treat each side of $(2.12)$ as $m$-th component of a vector, we  obtain by using our operation  \begin{theorem}   For all nonnegative integers m
 	\begin{equation} \delta_{m,0}-\frac{(-1)^{m}}{m+1}=\sum_{i=0}^{m}(1-\delta_{m,0})(-1)^{i}\binom{m}{i}\zeta(i-m), \end{equation}
 which is equivalent to \begin{equation}
 \boldsymbol{id}-(\boldsymbol{-1H})=\boldsymbol{-1H}\star(0,-1,0,0,\cdots)\star\boldsymbol{\zeta(-s)}.
 \end{equation}
\end{theorem}
\begin{proof} We just need to be clear that the right hand of $(2.14)$ is $e^{-z}-1$ multiplyed to the generating function of $\zeta(-s)$.
	\end{proof}
We do the convolution to both sides of  $(2.14)$ with  $\boldsymbol{-1B}$ , then
\begin{equation}\boldsymbol{-1B}-(\boldsymbol{id})=\boldsymbol{id}\star(0,-1,0,0,\cdots)\star\boldsymbol{\zeta(-s)}.\end{equation}
 Comparing their components, we have \begin{equation}(-1)^{n}B_n-\delta_{0,n}=-n\zeta(1-n).\end{equation}
Since $ B_{2k+1}=0, \ k>0$,  we get the well-known relaton $\zeta(1-n)=-\frac{B_n}{n},\ n>0$.
\subsection{Hurwitz zeta function}
This operation can be used for Hurwitz zeta function to deduce a relation with Bernoulli polynomial. For $0<a\leq1,$ the Hurwitz zeta function $\zeta(s,a)$ is defined by the series \[\zeta(s,a)=\sum_{n=0}^{\infty}\frac{1}{(n+a)^{s}},\] which converges absolutely for $Re(s)>1$. We use the same trick on $\zeta(s)$ to obtain (see\cite{3}) \begin{equation}
\zeta(s,a)=a^{-s}+\frac{a^{1-s}}{s-1}-\sum_{r=1}^{m}\frac{s^{(r)}}{(r+1)!}(\zeta(s+r,a)-a^{-s-r})+\frac{s^{(m+1)}}{(m+1)!}\sum_{n=0}^{\infty}\int_{0}^{1}\frac{u^{m+1}}{(u+n+a)^{s+m+1}}
\end{equation} where $s^{(n)}$ is rising factorial defined as $s^{(n)}=s(s+1)\cdots(s+n-1)$. Putting $s=1-m$, and  noting for $r=m$ that $\zeta(s+r,a)$ has a simple pole at $s=1-m$, we obtain \begin{equation}
\zeta(1-m,a)=a^{m-1}-\frac{a^{m}}{m}+\frac{(-1)^{m}}{m(m+1)}-\sum_{r=1}^{m-1}(-1)^{r}\frac{(m-1)(m-2)\cdots(m-r)}{(r+1)!}(\zeta(s+r,a)-a^{-s-r}).
\end{equation} 
Multiplying $m$ on both sides, we get \begin{equation}
m\zeta(1-m,a)=\frac{(-1)^{m}}{m+1}+\sum_{r=1}^{m-1}(-1)^{r+1}\binom{m}{r+1}(\zeta(-m+r+1,a)-\sum_{v=0}^{m}\binom{m}{v}(-1)^{v}a^{m-v}.
\end{equation} 
Therefore we have \begin{theorem} For integers $m>0$, \begin{equation}
\frac{(-1)^{m+1}}{m+1}=\sum_{v=1}^{m}(-1)^{v}\binom{m}{v}\zeta(v-m,a)-(a-1)^{m}.\end{equation} \end{theorem}
Using the same method from theorem 2.2, also noting that $\zeta(0,a)=1-\frac{3a}{2},$  we have \begin{equation}-\boldsymbol{-1H}=(0,-1,1,-1,\cdots)\star\boldsymbol{\zeta(-s,a)}-\boldsymbol{(a-1)},\end{equation} where $\boldsymbol{\zeta(-s,a)}=(\zeta(0,a),\zeta(-1,a),\zeta(-2,a),\cdots)$.
Doing convolution to both sides of $(2.21)$ with $\boldsymbol{-1B}$, together with 
 \[\boldsymbol{-1B}=\frac{-z}{e^{-z}-1}=\frac{ze^{z}}{e^{z}-1},\ \boldsymbol{(a-1)}=e^{(a-1)z},\] 
  we have \begin{equation}
\boldsymbol{B(a)}-\boldsymbol{id}=(0,-1,0,0,\cdots)\star\boldsymbol{\zeta(-s,a)}.
\end{equation} Comparing their components we get for $n>0$, 
\begin{equation} 
-n\zeta(1-n,a)=B_n(a).
\end{equation} \\

\section{Forward difference}
Recall that forward difference is $\Delta_h[f](x)=f(x+h)-f(x)$. When $h=1$, we denote it as 
\begin{equation}
\Delta[f](x)=f(x+1)-f(x).
\end{equation} 
We  can also deduce the $n$-th order forward difference by induction, that is,
\[
\Delta^{n}[f](x)=\Delta\{\Delta^{n-1}[f](x)\}=\Delta^{n-1}[f](x+1)-\Delta^{n-1}[f](x).\]
Then we get 
\begin{equation}
\Delta^{n}[f](x)=\sum_{i=0}^{n}\binom{n}{i}(-1)^{n-i}f(x+i).
\end{equation}
Look at the above equation. If we define a sequence as $(\Delta^{0}[f](x),\Delta^{1}[f](x),\cdots,\Delta^{n}[f](x),\cdots)$, denoted $\boldsymbol{\Delta f}$, we can connect it with our operation $\star$ and get 
\begin{equation}
\boldsymbol{\Delta f}=(f(x),f(x+1),\cdots,f(x+n),\cdots)\star(1,-1,\cdots,(-1)^{n},\cdots).
\end{equation}

Now consider $f(x)$ to be $\Gamma(x+1)(x+1)^{-s}$, we have 
\begin{align}
&(\Gamma(x+1)(x+1)^{-s},\Delta\Gamma(x+1)(x+1)^{-s},\cdots,\Delta^{n}\Gamma(x+1)(x+1)^{-s},\cdots)\\
&=(\Gamma(x+1)(x+1)^{-s},\cdots,\Gamma(x+n)(x+n)^{-s},\cdots)\star(1,\cdots,(-1)^{n},\cdots).\end{align}
Notice that $ (1,\cdots,(-1)^{n},\cdots)=e^{-z}$. If we let $x=0$ and $z=1$, then 
\begin{equation}
\sum_{n=0}^{\infty}\frac{\Delta^{n}\Gamma(1)(1)^{-s}}{n!}=e^{-1}\sum_{n=0}^{\infty}\frac{1}{n^{s}}.
\end{equation} 
This is equvalent to 
\begin{theorem}
\begin{equation}
\zeta(s)=e\sum_{n=0}^{\infty}\sum_{i=0}^{n}(-1)^{n-i}\frac{1}{(n-i)!(1+i)^{s}}. 
\end{equation}\\
\end{theorem}

\section{Calculus}
In this section it comes to my favorite part, where we  build a new conception about derivation and integral.

\subsection{Derivation and integral}
We learn from elementary calculus that derivation and integral are mutually inverse operations. Amazingly, using our conception of vector space, we will redefine derivation and integral.

If an analytic function $f(z)=\sum_{n=1}^{\infty}a_n\frac{z^{n}}{n!}$ is denoted $(0,a_1,a_2,a_3,\cdots)$, then \[ f'(z)=\sum_{n=0}^{\infty}a_{n+1}\frac{z^{n}}{n!}=(a_1,a_2,a_3,a_4,\cdots),\] and \[\int_{\gamma}f(z)=\sum_{n=1}^{\infty}\int_{\gamma}a_{1}\frac{z^{n}}{n!}=\sum_{n=2}^{\infty}a_{n-1}\frac{z^{n}}{n!}=(0,0,a_0,a_1,a_2,\cdots),\] 
where $\gamma$ is a smooth curve connecting $0$ and $z$ in the domain of $f(z)$. Now the reader may discover the fact that is derivation and integral are just like moving components left and right!
	
 Paper \cite{2} shows an interesting relation which is \begin{equation}
 	\int_{0}^{1}S_n(x)=\zeta(-n), n\in \mathbb{N},
 \end{equation} where $S_n(m)=\sum_{j=1}^{n-1}j^{m}$. We will explain this relation in the rest of this paper.

 Using Faulhaber’s Formula, we have \[S_n(x)=\frac{1}{n+1}\sum_{k=0}^{n}\binom{n+1}{k}(-1)^{k}B_k(x-1)^{n-k+1}.\] 
So we get
 \begin{equation}S_n(x)=\sum_{k=0}^{n}\frac{n!}{k!}(-1)^{k}B_k\frac{(x-1)^{n+1-k}}{(n+1-k)!}.\end{equation} 
 Considering the Taylor series of an analytic function at $1$, we use the set $\{\frac{(z-1)^{k}}{k!}\mid k=0,1,2,\cdots\}$ as a basis. We can denote the right side of $(4.2)$ as \[(0,(-1)^{n}B_n,(-1)^{n-1}nB_{n-1},\cdots,n!B_0,0,0,\cdots).\]
 Then $\int_{0}^{x}S_n(x)$ indicates moving one step to the right, i.e. \begin{equation}(0,0,(-1)^{n}B_n,(-1)^{n-1}nB_{n-1},\cdots,n!B_0,0,0,\cdots).\end{equation} 
 Still we have \[\int_{0}^{1}\frac{(x-1)^{k}}{k!}=-\frac{(-1)^{k+1}}{(k+1)!}.\] 
 Combining this with the vector (4.3) we get \begin{align*}\int_{0}^{1}S_n(x)=\frac{(-1)^{n+1}}{(n+1)(n+2)}\sum_{k=0}^{n}\binom{n+2}{k}B_k&=\frac{(-1)^{n}(n+2)B_{n+1}}{(n+1)(n+2)} \\ &=\frac{(-1)^{n}B_{n+1}}{n+1}=\zeta(-n),\end{align*}when $n\in \mathbb{N}$.

Of course there are other proofs of this relation, but we are using this to point out the new conception of derivation and integral.

\begin{remark}
The reader may think about it  why we need the first component of Taylor series to be zero to induce the conception of derivation and integral.
\end{remark}

\bibliographystyle{plain}

\end{document}